\documentclass{article}
\usepackage{graphicx} % Required for inserting images
\usepackage{mystyle}
\title{Ultrapowers of determinacy models\\ as iteration trees on $\text{HOD}$}
\author{Gabriel Goldberg\footnote{
     The first author's research was partially supported by NSF Grant 2401789.
}\and Grigor Sargsyan\thanks{The second author's work is funded by the National Science Centre, Poland under the Weave-UNISONO call in the Weave Programme, registration number UMO-2021/03/Y/ST1/00281 and UMO-2023/05/Y/ST1/00194.}\and Benjamin Siskind\thanks{The third author's research was funded in whole or in part by the Austrian Science Fund (FWF) [10.55776/Y1498, 10.55776/PAT2418625].}}
\date{}

\begin{document}

\maketitle
\begin{abstract}
    In the 1990s, Steel and Woodin showed that under large cardinal hypotheses, the HOD of $L(\mathbb R)$ admits a fine-structural analysis. Although this theorem sheds light on various problems in descriptive set theory, the fine-structural representations of many fundamental objects of determinacy theory are still unknown. For example, Woodin asked whether the ultrapower of HOD by the closed unbounded filter on $\omega_1$ is given by an iteration tree on HOD according to its fine-structural extender sequence and canonical iteration strategy. In this paper, we give a positive answer to Woodin's question, not only for the closed unbounded filter but for any ultrafilter on an ordinal. The key tool that enables the solution of Woodin's problem is a recent advance in inner model theory: the Steel--Schlutzenberg theory of normalizing iteration trees, which allows us to represent HOD and its ultrapowers as normal iterates of a single countable mouse. Despite our results, the precise structure of the iteration trees that lead from HOD into its ultrapowers remains a mystery.
\end{abstract}
\section{Introduction}
There is a large body of work exploring connections between determinacy and large cardinals. This work began with Solovay’s discovery that the Axiom of Determinacy implies that the club filter on $\omega_1$ is an ultrafilter. Of course, this ultrafilter is ordinal definable, so it follows that $\omega_1$ is a measurable cardinal in $\HOD$. This is typical of how one shows that some large cardinal property is realized in $\HOD$ in the determinacy context: one produces an ultrafilter $U$ on an ordinal such that the ostensibly external embedding $i_U\restriction\HOD$ witnesses the appropriate large cardinal property. By a theorem of Kunen, \textit{any} ultrafilter on an ordinal is ordinal definable, so $i_U\restriction\HOD$ is actually an internal elementary embedding of $\HOD$, and so it really does realize the desired large cardinal property. This method was pushed further by Martin, Steel, and Woodin, culminating in Woodin’s result that in $L(\mathbb{R})$ under determinacy, $\Theta$ is a Woodin cardinal of $\HOD$, the witnesses to Woodinness coming from cleverly constructed ultrafilters on ordinals. 

Connections between determinacy and inner model theory have provided a different kind of understanding of the large cardinal structure of $\HOD$ in the determinacy context. Work of Steel and Woodin, and subsequenlty Sargsyan, Trang, and others, have provided fine-structural analyses of $\HOD$ in all known models of $\AD + V=L(P(\mathbb{R}))$; that is, they showed that in all these various models of determinacy, $\HOD|\Theta$ is a premouse (of some variety). Of course, this immediately implies that $\HOD|\Theta$ has all of the nice properties that premice have, for example GCH or the Ultrapower Axiom. But this also identifies a distinguished sequence of extenders of HOD which can be used to form fine-structural iteration trees (along with the resulting iterates and iteration maps).

In this paper we establish a close connection between ultrafilters on ordinals and this fine-structural understanding of $\HOD$ in the determinacy context. We show that in many of the known models of $\AD+V=L(P(\mathbb{R}))$, for $U$ an ultrafilter on an ordinal, $i_U(\HOD)$ is an iterate of $\HOD$ via an iteration tree coming from the distinguished extender sequence of $\HOD$ and $i_U\restriction \HOD$ is the corresponding iteration map. In particular, this holds in $L(\mathbb{R})$ under determinacy, answering a question of Woodin. 

There are essentially two ingredients to our proofs: the analysis of $\HOD$ in determinacy models, mentioned above, and full normalization, a more recent inner-model-theoretic tool developed in \cite{Farmer} and \cite{fullnormalization}. While both of these ingredients are quite involved technically, we will get to use both off the shelf, and our proofs are fairly short and simple by inner model theory standards.

We will review what we need from the $\HOD$ analysis and full normalization in \cref{section - preliminaries} before establishing our main results  in two contexts: under $\AD_\mathbb{R}+V=L(P(\mathbb{R}))+\textnormal{HPC}$ in \cref{section - AD_R}, and under $\AD+V=L(\mathbb{R})$ in \cref{section - L(R)}.
\section{Preliminaries}\label{section - preliminaries}

\subsection{The HOD analysis}

In this section we will collect some terminology and results about mouse pairs and the HOD analysis in the contexts of interest. By a \textit{premouse} we mean any one of three varieties:  an ms-indexed pure-extender premouse, a pfs pure-extender premouse, or a least branch strategy premouse. A \textit{partial iteration strategy} $\Sigma$ for a premouse $P$ is a partial strategy for choosing cofinal well-founded branches through stacks of normal trees on $P$; more precisely, $\Sigma$ is a partial function with domain some set $D$ of stacks of normal trees $\vec{\tree{S}}$ on $P$ such that
\begin{enumerate}
    \item $\lh(\vec{\tree{S}})$ is successor ordinal $\alpha+1$, $\lh(\tree{S}_\alpha)$ is a limit ordinal, and $\Sigma(\vec{\tree{S}})$ is a cofinal well-founded branch of $\tree{S}_\alpha$, and
    \item for any $\alpha<\lh(\vec{\tree{S}})$, and any limit $\lambda<\lh(\tree{S}_\alpha)$, $\vec{\tree{S}}\restriction\alpha^\frown \langle \tree{S}_\alpha\restriction\lambda\rangle\in D$ and $[0,\lambda)_{\tree{S}_\alpha}=\Sigma(\vec{\tree{S}}\restriction\alpha^\frown \langle \tree{S}_\alpha\restriction\lambda\rangle)$.
\end{enumerate}
We say a stack $\vec{\tree{S}}$ is \textit{by} $\Sigma$ if for all $\alpha<\lh(\vec{\tree{S}})$, and limit $\lambda<\lh(\tree{S}_\alpha)$, $\vec{\tree{S}}\restriction\alpha^\frown \langle \tree{S}_\alpha\restriction\lambda\rangle\in D$ and $[0,\lambda)_{\tree{S}_\alpha}=\Sigma(\vec{\tree{S}}\restriction\alpha^\frown \langle \tree{S}_\alpha\restriction\lambda\rangle)$.
We also say a stack of normal trees $\vec{\tree{S}}$ on $P$ is on $(P,\Sigma)$ if it is by $\Sigma$.

We will use a variation of the notion of mouse pair from \cite[Section 9.2]{steel-book}. 

\begin{defn}\label{defn - mouse pair}
  A \textit{mouse pair with scope} $X$ is a pair $(P,\Sigma)$ such that $P$ is premouse and $\Sigma$ is a partial iteration strategy for $P$ with domain $D\subseteq X$ such that \begin{enumerate}
      \item \begin{enumerate}
          \item if $\vec{\tree{S}}\in X$ is a stack of normal trees on $P$ of successor length $\alpha+1$ such that $\lh(\tree{S}_\alpha)$ is a limit ordinal and $\vec{\tree{S}}$ is by $\Sigma$, then $\vec{\tree{S}}\in D$,
      \item if $\vec{\tree{S}}$ is a stack of normal trees on $P$ by $\Sigma$ of limit length, then $M_\infty^{\vec{\tree{S}}}$ is well-founded, and
      \item if $\vec{\tree{S}}$ is a stack of normal trees on $P$ by $\Sigma$ and $\tree{T}$ a putative normal iteration tree on $M_\infty^{\vec{\tree{S}}}$ of successor length such that $\vec{\tree{S}}^\frown\langle\tree{T}\restriction \alpha\rangle$ is by $\Sigma$ for all $\alpha+1<\lh(\tree{T})$, then $M_\infty^\tree{T}$ is well-founded,
      \end{enumerate}
      \item $\Sigma$ is internally lift-consistent, push-forward consistent, fully normalizes well, and has very strong hull condensation, and
      \item if $P$ is a least branch premouse, then $(P,\Sigma)$ moves itself correctly.
  \end{enumerate}
\end{defn}

 The reader can find the several terms we have not defined in \cite{steel-book} and \cite{fullnormalization}, but we believe familiarity with these terms is not really necessary for understanding the paper. We only need a few facts about mouse pairs which we will state and discuss in these preliminary sections. 
 
 We have deviated from the definition of mouse pair in \cite{steel-book} in a couple of ways, mostly as a matter of convenience. First, we have restricted our iteration strategies to stacks of normal trees, whereas the iteration strategies in \cite{steel-book} act on a wider class of stacks of trees. Second, we have replaced quasi-normalizing well and strong hull condensation with fully normalizing well and very strong hull condensation. The main theorem of \cite{fullnormalization} is that if $(P,\Sigma)$ is a mouse pair the sense of \cite{steel-book} fully normalizing well and has very strong hull condensation, so that $(P,\bar{\Sigma})$ is a mouse pair in our sense, where $\bar{\Sigma}$ is the restriction of $\Sigma$ to stacks of normal trees. (Moreover, it can be shown by the methods of \cite{fullnormalization} that mouse pairs in our sense extend uniquely to mouse pairs in the sense of \cite{steel-book}, but we will not use this.) 

\begin{defn}\label{defn - sigma plus}
    Let $(P,\Sigma)$ be a mouse pair with scope $\textnormal{HC}$. A normal iteration tree $\tree{T}$ on $P$ is \textit{by} $\Sigma^+$ if every countable weak hull of $\tree{T}$ is by $\Sigma$.
\end{defn}

Note that for $\tree{T}$ a tree on $P$ of limit length by $\Sigma^+$, there is at most one branch $b$ of $\tree{T}$ such that $\tree{T}^{\frown} b$ is by $\Sigma^+$. In this case, we define $\Sigma^+(\tree{T})=b$. Since $\Sigma$ has very strong hull condensation, if $\tree{T}$ is by $\Sigma$, it is also by $\Sigma^+$. So $\Sigma^+$ is a partial iteration strategy for $P$ extending $\Sigma$ to certain iteration trees of uncountable length. 

We also say that a stack of trees $\langle \tree{S}, \tree{T}\rangle$ is by $\Sigma^+$ if $X(\tree{S},\tree{T}\restriction \xi+1)$ is by $\Sigma^+$ for all $\xi<\lh(\tree{T})$. if $\tree{S}$ is by $\Sigma^+$ with last model $Q$, we let $\Sigma^+_Q$ be the resulting tail strategy; that is, $\tree{T}$ is by $\Sigma^+_Q$ if and only if $\langle \tree{S}, \tree{T}\rangle$ is by $\Sigma^+$.

\begin{lma}[{Steel}, \cite{mousepairs}]\label{lemma - strategy extension}
    Assume $\AD^+$. Let $(P, \Sigma)$ be a least branch hod pair with scope $\textnormal{HC}$. Then $\Sigma^+$ restricts to a total iteration strategy for normal trees of length less than $\Theta$. That is, whenever $\tree{T}$ is by $\Sigma^+$ and has limit length less than $\Theta$, $\Sigma^+(\tree{T})$ is defined.
\end{lma}

Therefore, under $\AD^+$, if $(P,\Sigma)$ is a least branch hod pair with scope $\textnormal{HC}$, then $(P, \Sigma^+\restriction V_\Theta)$ is a least branch hod pair with scope $V_\Theta$. 

In this paper the only determinacy models we will consider are $L(\mathbb{R})$ and those satsifying $\AD_\mathbb{R}+V=L(P(\mathbb{R}))+\textnormal{HPC}$. The HOD analysis has been carried out in both contexts. In $L(\mathbb{R})$, this is due to Steel and Woodin; under $\AD_\mathbb{R}+V=L(P(\mathbb{R}))+\textnormal{HPC}$, this is due to Steel. 

The following is the main theorem of \cite[\S 7]{hodas} (though not explicitly stated in this form).
\begin{thm}[Steel-Woodin, \cite{hodas}]\label{thm - hod L(R)}
    Assume $\AD+ V=L(\mathbb{R})$. Then there is a ms-indexed pure-extender premouse $H$ such that $V_\Theta\cap \HOD$ is the universe of $H|\Theta$ and there is a partial iteration strategy $\Lambda$ for $H$ and $\HOD=L[H, \Lambda]$.
\end{thm}

The following is one of the main theorems of \cite{mousepairs}.

\begin{thm}[{Steel, \cite{mousepairs}}]\label{thm - steel hod analysis}
    Assume $\AD_\mathbb{R}+V=L(P(\mathbb{R}))+  \textnormal{HPC}$. Then there is a least branch premouse $H$ such that $V_\Theta\cap \HOD$ is the universe of $H$ and $\HOD=L[H]$.
\end{thm}

Unfortunately, we won't be able to simply quote these concise expressions of the HOD analyses in our proofs. To prove our theorem in $L(\mathbb{R})$, we will actually need the reflection argument used in the proof of \cref{thm - hod L(R)}. To prove our theorem under $\AD_\mathbb{R}+V=L(P(\mathbb{R}))+  \textnormal{HPC}$, we will need a refinement of \cref{thm - steel hod analysis}, which we will now state.
\begin{defn}\label{defn - cutpoint}
    For $(P,\Sigma)$ and $(Q,\Lambda)$ mouse pairs of the same type, we let $(P,\Sigma)\mo(Q,\Lambda)$ if $P\mo Q$ and $\Lambda_P=\Sigma$. We let $(P,\Sigma)\mo^*(Q,\Lambda)$ if $P$ is also a strong cutpoint initial segment of $Q$; that is, $P$ is passive $Q$ has no extenders overlapping $o(P)$.
\end{defn}

\begin{thm}[{Steel, \cite{mousepairs}}]\label{thm - pc generators}
    Assume $\AD_\mathbb{R}$ and  \textnormal{HPC}. Then there is a sequence $\langle (H_\alpha,\Sigma_\alpha)\mid\alpha<\eta\rangle$ of least branch hod pairs such that for every $\alpha<\eta$,
    \begin{enumerate}
    \item  the universe of $H_\alpha$ is $\HOD\cap V_{o(H_\alpha)}$ and $o(H_\alpha)$ is strongly inaccessible in $\HOD$ and closed under ultrapowers in $V$,
        \item\label{item - pc generators}  there is a countable least branch hod pair $(P, \Sigma)$ with scope $\textnormal{HC}$ such $(H_\alpha, \Sigma_\alpha)$ is an iterate of $(P, \Sigma^+\restriction V_\Theta)$ and $H_\alpha=M_\infty(P,\Sigma)$,
        \item for every $\alpha<\beta<\eta$, $(H_\alpha,\Sigma_\alpha)\mo^* (H_\beta, \Sigma_\beta)$,
        %define cutpoint initial segment
    \end{enumerate}
Finally, letting $H$ be the least branch premouse $\bigcup_{\alpha<\eta} H_\alpha$, we have $o(H)=\Theta$ and $L[H]=\HOD$.
\end{thm}

This analysis of $\HOD$ determines a natural partial iteration strategy for normal iteration trees on $\HOD$ viewed as least branch premouse which we denote $\Sigma_\HOD$. A normal tree $\tree{T}$ is by $\Sigma_\HOD$ if for any limit ordinal $\lambda< \lh(\tree{T})$, there is an $\alpha<\eta$ such that $\tree{T}\restriction\lambda$ is based on $H_\alpha$ and $[0,\lambda)_\tree{T}=\Sigma_\alpha(\tree{T}\restriction\lambda)$.

\subsection{Full normalization}
The full normalization of a stack of normal trees $\vec{\tree{S}}$ on a mouse pair $(P,\Sigma)$ is defined by recursion using the normalization process for stacks of length two at successor stages and taking direct limits at limit stages. One of the main theorems of \cite{Farmer} and \cite{fullnormalization} is that if $(P,\Sigma)$ is a mouse pair with very strong hull condensation, this process does not break down and produces a single normal tree $X(\vec{\tree{S}})$ on $(P,\Sigma)$ with the same last model and main branch embedding as $\vec{\tree{S}}$ (in the case these are defined).
We will only need to consider stacks of length at most $\omega$ in this paper, so we briefly discuss the length two case and the direct limit process for the special case of stacks of length $\omega$.

First, given a stack of length two $\langle\tree{S},\tree{T}\rangle$ on $(P,\Sigma)$ such that $\tree{T}$ has successor length, the full normalization $\tree{X}=X(\tree{S},\tree{T})$ is a single normal tree on $(P,\Sigma)$ such that $M_\infty^\tree{X}=M_\infty^\tree{T}$. If $\tree{T}$ does not drop along its main branch, the normalization process also produces a weak tree embedding $\Phi:\tree{S}\to \tree{X}$, a certain kind of system of embeddings which embeds the iteration tree structure of $\tree{S}$ into that of $\tree{X}$.  
The definition of this weak tree embedding is quite involved but we will need very little about it, which we collect below. We refer the reader to \cite{Farmer} and \cite{fullnormalization} for further details. If $\tree{S}$ also doesn't drop along its main branch, then we additionally have that $i^\tree{T}_{0,\infty}\circ i^\tree{S}_{0,\infty}=i^\tree{X}_{0,\infty}$.

One component of the weak tree embedding $\Phi:\tree{S}\to \tree{X}$ is the $u$-map $u^\Phi$, an injective map from $\lh(\tree{S})$ into $\lh(\tree{X})$. Roughly, the $u$-map keeps track of an association between the exit extenders of $\tree{S}$ and those of $\tree{X}$, determined by the rest of $\Phi$. We'll only use a special case of this association, recorded in the following lemma.
\begin{lma}\label{lemma - tree emb}
Suppose $\langle\tree{S},\tree{T}\rangle$ is a stack of normal trees with a last model on a mouse pair $(P,\Sigma)$. Assume that $\tree{T}$ doesn't drop along its main branch. Let $\tree{X}=X(\tree{S},\tree{T})$ and $\Phi:\tree{S}\to \tree{X}$ be the associated weak tree embedding. Suppose that $\alpha+1<\lh(\tree{X})$ is such that $\lh(E_\alpha^\tree{X})>\lh(E_\xi^\tree{T})$ for all $\xi+1<\lh(\tree{T})$. Then $\alpha\in \ran(u^\Phi)$. Moreover, letting $\bar{\alpha}=(u^\Phi)^{-1}(\alpha)$, $i^{\tree{T}}_{0,\infty}$ restricts to a cofinal elementary map from $M_{\bar{\alpha}}^\tree{S}|\lh(E_{\bar{\alpha}}^\tree{S})$ into $M_{{\alpha}}^\tree{X}|\lh(E_{{\alpha}}^\tree{X})$.
\end{lma}

If $\langle\tree{S},\tree{T}\rangle$ is a stack of normal trees with a last model on a least branch hod pair $(P,\Sigma)$ and $\tree{T}$ \textit{does} drop along its main branch, then the exit extenders of $\tree{T}$ are used cofinally in $\tree{X}=X(\tree{S},\tree{T})$ so there can be no $\alpha+1<\lh(\tree{X})$ is such that $\lh(E_\alpha^\tree{X})>\lh(E_\xi^\tree{T})$ for all $\xi+1<\lh(\tree{T})$. Therefore the assumption in \cref{lemma - tree emb} that $\tree{T}$ does not drop causes no real loss of generality. (However this assumption is needed to ensure that there is a \textit{total} weak tree embedding from $\tree{S}$ into $\tree{X}$.)

We need another basic fact about normalizing stacks of length two.

\begin{lma}\label{lemma - drops}
 Suppose $\langle\tree{S},\tree{T}\rangle$ is a stack of normal trees with a last model on a mouse pair $(P,\Sigma)$ and $\tree{T}$ does drop along its main branch. Let $\tree{X}=X(\tree{S},\tree{T})$ and $\Phi:\tree{S}\to \tree{X}$ be the associated weak tree embedding. Then every exit extender of $\tree{T}$ appears in $\tree{X}$ and such extenders appear cofinally in $\tree{X}$. As a consequence, there can be no $\alpha+1<\lh(\tree{X})$ is such that $\lh(E_\alpha^\tree{X})>\lh(E_\xi^\tree{T})$ for all $\xi+1<\lh(\tree{T})$. Moreover, for all $\xi+1<\lh(\tree{T})$, letting  $\alpha+1$ be such that $E_{\alpha}^{\tree{X}}=E_\xi^\tree{T}$, $\alpha\not\in \ran(u^\Phi)$. 
\end{lma}

Finally, we need to look more closely at normalizing a stack of length two. To normalize $\langle\tree{S}, \tree{T}\rangle$, we proceed by induction on $\lh(\tree{T})$, forming the auxiliary normalizations of $\langle \tree{S}, \tree{T}\restriction \xi+1\rangle$ and associated weak tree embeddings between these trees along the tree-order of $\tree{T}$. At limit stages of $\tree{T}$, we take direct limits corresponding to the branch choices of $\tree{T}$. The following technical lemma shows that we can actually recover the branch choices of $\tree{T}$ from the branches chosen in the direct limit tree (or, more importantly, in the final full normalization of  $\langle\tree{S}, \tree{T}\rangle$).\footnote{In \cite[Section 6.6]{steel-book}, the analogous statement is established for embedding normalization, but the same arguments work for full normalization.} 
\begin{prp}[{Steel \cite[Section 6.6]{steel-book}, Schlutzenberg \cite{fullnormalization}}]\label{prop - branch uniqueness}
    Suppose $\langle\tree{S},\tree{T}\rangle$ is a stack of normal trees on a premouse $P$ and that $\tree{T}$ has limit length. For $\xi<\lh(\tree{T})$, let $\tree{X}_\xi=X(\tree{S},\tree{T}\restriction \xi+1)$ and let $\alpha_\xi+1<\lh(\tree{X}_\xi)$ be such that $E_{\alpha_\xi}^{\tree{X}_\xi}=E_\xi^\tree{T}$. Let $\tree{X}=\bigcup_{\xi<\lh(\tree{T})} \tree{X}_\xi\restriction\alpha_\xi+1$. Then for any cofinal wellfounded branch $b$ of $\tree{X}$, there is a unique cofinal branch $c$ of $\tree{T}$ such that $\tree{X}^\frown b\moq X(\tree{S},\tree{T}^\frown c)$.
\end{prp}
%Here you cite nitcis, but do we want steel-book instead? Also I think we should say where it appears...

Note that we do not assume that $c$ is a wellfounded branch of $\tree{T}$. If $M_c^\tree{T}$ is illfounded, then $X(\tree{S},\tree{T}^\frown c)$ is not an iteration tree (and may not even be a putative iteration tree). The conclusion asserts that $X(\tree{S},\tree{T}^\frown c)\restriction\sup_\xi\alpha_\xi+1$ is an iteration tree, however.

Next we consider normalizing a stack of length $\omega$. Let $\vec{\tree{S}}$ be a non-dropping length $\omega$ stack of normal trees on a mouse pair $(P,\Sigma)$.
Let $\tree{X}_{n}=X(\tree{S}_0,\ldots, \tree{S}_n)$. Then for $n< m$, $\tree{X}_m=X(\tree{X}_n, \vec{\tree{S}}\restriction (n, m])$ and so there is a resulting weak tree embedding $\Phi_{n,m}:\tree{X}_n\to \tree{X}_m$. We also set $\Phi_{n,n}$ to be the identity weak tree embedding from $\tree{X}_n$ to itself. One can show that if $l\leq m\leq n$, then $\Phi_{l, n}=\Phi_{m, n}\circ \Phi_{l, m}$. The full normalization $X(\vec{\tree{S}})$ is the direct limit of the linear system $\langle \tree{X}_n, \Phi_{n,m}\mid n\leq m<\omega\rangle$. As mentioned above, one of the main theorems of \cite{fullnormalization} and \cite{Farmer} gives that $X(\vec{\tree{S}})$ is a normal tree on $(P,\Sigma)$ with last model $M_b^{\vec{\tree{S}}}$ for $b$ the unique cofinal branch of $\vec{\tree{S}}$. Moreover, there are resulting direct limit weak tree embeddings $\Phi^{\vec{\tree{S}}}_{n,\omega}:\tree{X}_n\to X(\vec{\tree{S}})$.

We will only need to use a couple additional facts about this process. First, every node in the full normalization $X(\vec{\tree{S}})$ is in the range of the $u$-map of some direct limit weak tree embedding.

\begin{prp}\label{prop - infinite stack}
    Let $\vec{\tree{S}}$ be a non-dropping length $\omega$ stack of normal trees on a mouse pair $(P,\Sigma)$ with scope $\textnormal{HC}$. 
   For every $\alpha+1<\lh(X(\vec{\tree{S}}))$, there exist an $n<\omega$ and $\bar{\alpha}+1<\lh(\tree{X}_n)$ such that $u^{\Phi^{\vec{\tree{S}}}_{n,\omega}}(\bar{\alpha})=\alpha$. 
   %and $t^{\Phi_{n,\infty}}_{\bar{\alpha}}(E_{\bar{\alpha}}^{\tree{X}_n})=E_\alpha^{\tree{X}_\infty}$ 
\end{prp}

Second, a version of the commutativity of the associated weak tree embeddings passes through limits.

\begin{prp}\label{prop - infinite stack commutativity}
    Let $\vec{\tree{S}}$ be a non-dropping length $\omega$ stack of normal trees on a mouse pair $(P,\Sigma)$ with scope $\textnormal{HC}$. 
    %For $n\leq m\leq l<\omega$, let $P_n=M_0^{\tree{S}_n}$, $\tree{X}^n_m=X(\tree{S}_n,\ldots, \tree{S}_m)$, and $\Phi^n_{m,l}$ be the associated weak tree embedding from $\tree{X}^n_m$ into $\tree{X}^m_l$. 
   % Then for every $n<\omega$, the direct limit $\tree{X}^n_{\infty}$ of $\langle \tree{X}^n_m, \Phi^n_{m,l}\mid n\leq m\leq l<\omega\rangle$ is a non-dropping normal tree on $(P_n,\Sigma_{P_n})$ and has last model $M_b^{\vec{\tree{S}}}$ for $b$ the unique cofinal branch of $\vec{\tree{S}}$. 
    For $m\leq n\leq \omega$, let $\tree{X}_{m,n}=X(\vec{\tree{S}}\restriction[m,n))$. Then $\tree{X}_{0,\omega}=X(\tree{X}_{0,n}, \tree{X}_{n,\omega})$ and $\Phi^{\vec{\tree{S}}}_{n,\omega}$ is the weak tree embedding from $\tree{X}_{0,n}$ into $X(\tree{X}_{0,n}, \tree{X}_{n,\omega})$ associated to this normalization.
\end{prp}

Finally, we mention a couple consequences of full normalization which we will use. First, we note what is probably the most important consequence:  \textit{positionality}, another regularity property for iteration strategies.
\begin{lma}[{Positionality, Steel \cite[Section 5.2]{steel-book}}]\label{lemma - commutativity}
   Suppose $\vec{\tree{S}}_0$ and $\vec{\tree{S}}_1$ are non-dropping stacks of normal trees on a mouse pair $(P,\Sigma)$ with a common last model. Then $i^{\vec{\tree{S}}_0}_{0,\infty}=i^{\vec{\tree{S}}_1}_{0,\infty}$.
\end{lma}

This is actually an immediate consequence of full normalization and an essentially trivial instance of positionality: non-dropping normal trees $\tree{S}_0$, $\tree{S}_1$ on a mouse pair $(P,\Sigma)$ with a common last model $Q$ have the same iteration map because $\tree{S}_0=\tree{S}_1$, as both are just the result the normal tree on $(P,\Sigma)$ obtained by comparing $P$ and $Q$.

\begin{rmk}\label{rmk - uniqueness of normal trees} The fact just mentioned is totally general and will be used often: a normal tree $\tree{T}$ of successor length on a mouse pair $(P,\Sigma)$ is completely determined by its last model and, in fact, $\tree{T}$ is the tree on $(P,\Sigma)$ obtained by comparing $P$ and $M_\infty^\tree{T}$ by least extender disagreements. (In particular, we never encounter strategy disagreements and $M_\infty^\tree{T}$ doesn't move in this comparison.)
\end{rmk}

Second, we have a related directedness result for non-dropping iterates of a mouse pair.

\begin{lma}\label{lemma - directedness}
    Let $\delta$ be a regular cardinal. Suppose \((M, \Sigma_M)\) and \( (N, \Sigma_N)\)
    are non-dropping ${<}\delta$-iterates of a mouse pair \( (P, \Sigma)\) with scope $H_\delta$. Let $\mathcal{T}_M$ and $\mathcal{T}_N$ be the padded normal iteration trees on $(M, \Sigma_M)$ and $(N, \Sigma_N)$ obtained by comparison by least extender disagreement. Then the length $\eta$ of these trees is less than $\delta$, no strategy disagreements occur, and neither side drops along the main branch, so $M_\infty^{\mathcal{T}_M}=M_\infty^{\mathcal{T}_N}$ and $\Sigma_{M_\infty^{\mathcal{T}_M}}=\Sigma_{M_\infty^{\mathcal{T}_N}}$. Moreover, for all $\alpha+1<\eta$, either $E_\alpha^{\mathcal{T}_M}$ or $E_\alpha^{\mathcal{T}_N}$ is trivial.
\end{lma}

\section{Main theorems}

\subsection{$\AD_\mathbb{R}+V=L(P(\mathbb{R}))+\textnormal{HPC}$}\label{section - AD_R}

In this section we'll start by proving the our main theorem under $\AD_\mathbb{R}+V=L(P(\mathbb{R}))+\textnormal{HPC}$. We will also be able to generalize it, using a somewhat more complicated argument, to ultrafilters that are not on ordinals.

\begin{lma}[{\(\text{AD}^+ + V= L(P(\mathbb R))\)}]\label{lemma - generators}
    If \(U\) is an ultrafilter on \(X\leq^*\mathbb R\),
 then \(i_U\restriction \HOD = j_E^{\HOD}\)
    where \(E\) is an extender in \(\HOD\) of support
    \(\Theta\).
\end{lma}

\begin{proof}

    We first claim that every set of ordinals $S$ of size less than $\Theta$ is covered by a set $T\in \HOD$ of size less than $\Theta$. Since $V=L(P(\mathbb{R}))$, $S$ is definable from some set of reals $A$ by a formula $\varphi$; say $S=\{\alpha\in \HOD\mid \varphi(\alpha, A)\}$. Let $S_B=\{\alpha\in \HOD\mid \varphi(\alpha, B)\}$. Let $\xi$ be the Wadge-rank of $A$. Let $Z$ be the set of all sets of reals $B$ of Wadge-rank $\xi$ such $S_B$ has the same ordertype as $S$. Let $T=\{\alpha\in \HOD\mid \exists B\in Z\,\varphi(\alpha, B)\}$. It's easy to show that $T$ has size less than $\Theta$.

To prove the lemma, it suffices to show that for every ordinal $\nu$ there is a set $T\in \HOD$ of size less than $\Theta$ such that $\nu\in i_U(T)$, since then the measures derived from $i_U\restriction \HOD$ concentrate on sets of size less than $\Theta$. Fix $f:X\to \Ord$ such that $[f]_U=\nu$. Let $S=\ran(f)$ and let $T\in \HOD$ be a set of size less than $\Theta$ covering $S$. For all $x$, $f(x)\in S$, so by the definition of the ultrapower, $\nu\in i_U(S)$. By elementarity of $i_U\restriction \HOD_S$, $\nu\in i_U(T)$, proving the lemma.
    \end{proof}

Note that in the case that $U$ is an ultrafilter on an ordinal, the extender $E$ derived from $i_U\restriction \HOD$ with support $\Theta$ also has length $\Theta$, since $\sup i_U[\Theta]=\Theta$ in this case (since $\Theta$ is a strong limit).

\begin{thm}[Goldberg \cite{uniqueness}]\label{thm - uniqueness}
If $j_0$ and $j_1$ are elementary embeddings from $V$ into the same inner model, then $j_0\restriction \Ord=j_1\restriction \Ord$.
\end{thm}

Note that the statement of this theorem is not actually expressible in the language of set theory. For our purposes, the result should be construed as a theorem of $\ZFC$ in the language of set theory expanded by additional predicates for $j_0$ and $j_1$ (i.e. where replacement is stated in this expanded language). We will apply this theorem to structures of the form $(V_\kappa, \in, j_0, j_1)$ for $\kappa$ a strongly inaccesssible cardinal and $j_0$ and $j_1$ elementary embeddings from $V_\kappa$ into some transitive set $M\subseteq V_\kappa$. Such structures are easily seen to satisfy $\ZFC$ in this expanded language. %explain background theory 

\begin{lma}\label{lemma - iterate criterion}
    Suppose \((M, \Sigma_M)\) and \( (N, \Sigma_N)\)
    are non-dropping iterates of a mouse pair \( (P, \Sigma)\) such that $P$ satisfies $\ZFC$. If \( N\subseteq  M\), then \( (N, \Sigma_N)\) is a normal non-dropping iterate of \( (M, \Sigma_M)\).
\end{lma}

\begin{proof}
    Let $\mathcal{T}_M$, $\mathcal{T}_N$ be the (padded) trees of the comparison of $(M,\Sigma_M)$ with $(N, \Sigma_N)$ by least disagreement. By \cref{lemma - directedness}, $\mathcal{T}_M, \mathcal{T}_N$ don't drop and have a common last model $(Q, \Sigma_Q)$. Let $\mathcal{S}_M$, $\mathcal{S}_N$, and $\mathcal{S}_Q$ witness that $M$, $N$, and $Q$ are normal iterates of $(P,\Sigma)$. In particular, $\mathcal{S}_Q$ is the full normalization $X(\mathcal{S}_M, \mathcal{T}_M)=X(\mathcal{S}_N, \mathcal{T}_N)$.

    Towards a contradiction, suppose $\mathcal{T}_N$ is non-trivial. Let $\xi$ be least such that $E_\xi^{\mathcal{T}_N}$ is non-trivial. 
    Let $E=E_\xi^{\mathcal{T}_N}$ and $\alpha$ be least such that either $\alpha+1=\lh(\mathcal{S}_N)$ or $\lh(E_\alpha^{\tree{S}_N})>\lh(E)$. Also let $\tree{X}=X(\mathcal{S}_M, \mathcal{T}_M\restriction \xi+1)$ (ignoring the padding of $\mathcal{T}_M$).

    Then $\tree{X}\restriction\alpha+1 =\tree{S}_N\restriction\alpha+1$, since $M_\infty^\tree{X}$ and $N=M_\infty^{\tree{S}_N}$ agree up to $\lh(E)$. Moreover, $E=E^\tree{X}_\alpha$ since $\tree{S}_N\restriction\alpha+1^{\frown} \langle E\rangle$ is a normal tree by $\Sigma$ whose last model agrees with the last model of $\tree{X}$ strictly past $\lh(E)$. By \cref{lemma - drops}, $\tree{T}_M\restriction\xi+1$ cannot drop along its main branch. Let $\Phi$ be the weak tree embedding from $\tree{S}_M$ into $\tree{X}$. 
    By \cref{lemma - tree emb}, $\alpha=u^\Phi(\bar{\alpha})$ for some $\bar{\alpha}+1<\lh(\tree{S}_M)$ and $i=i^{\tree{T}_M}_{0,\xi}$ restricts to a cofinal elementary map from $M_{\bar{\alpha}}^{\tree{S}_M}|\lh(E_{\bar{\alpha}}^{\tree{S}_M})$ into $N|\lh(E)=M_\alpha^\tree{X}|\lh(E_\alpha^\tree{X})$. 
    
    Since $N\subseteq M$ and $M$ is closed under its iteration strategy $\Sigma_M$, $\mathcal{T}_M\restriction \xi+1\in M$. In particular, $i\restriction M_{\bar{\alpha}}^{\tree{S}_M}||\lh(E_{\bar{\alpha}}^{\tree{S}_M})$ is a member of $M$. But then $E_{\bar{\alpha}}^{\tree{S}_M}=i^{-1}(E)$ is a member of $M$, since $E\in N\subseteq M$. So there is an $M_{\bar{\alpha}}^{\tree{S}_M}|\lh(E_{\bar{\alpha}}^{\tree{S}_M})$-definable surjection from $\lambda(E_{\bar{\alpha}}^{\tree{S}_M})$ onto $\lh(E_{\bar{\alpha}}^{\tree{S}_M})$, contradicting that $\lh(E_{\bar{\alpha}}^{\tree{S}_M})$ is a cardinal of $M$, since $E_{\bar{\alpha}}^{\tree{S}_M}$ is used in the normal iteration from $P$ into $M$.
\end{proof}

\begin{lma}\label{lemma - strategy preservation}
    Assume $\AD^++V=L(P(\mathbb{R}))$. Let $(P,\Sigma)$ be a mouse pair with scope $\textnormal{HC}$ and $\tree{T}$ be a normal tree by $\Sigma^+$. (See \cref{defn - sigma plus}.) Let $U$ be an ultrafilter on a set $X\leq^*\mathbb{R}$ such that $\Ult(\Ord, U)$ is well-founded. Then $i_U(\tree{T})$ is by $\Sigma^+$.
\end{lma}

\begin{proof}
    We need to show that every countable weak hull of $i_U(\tree{T})$ is by $\Sigma$. So fix $\bar{\tree{T}}$ a countable weak hull of $i_U(\tree{T})$. We'll show that $\bar{\tree{T}}$ is a weak hull of $\tree{T}$. Fix $x\in\mathbb{R}$ such that $\bar{\tree{T}}\in \textnormal{HC}^{\HOD_x}$ and let $H=\HOD_{x, \tree{T}}$. Then $i_U$ restricts to an elementary embedding from $H$ into $i_U(H)$. Note that $i_U(\bar{\tree{T}})=\bar{\tree{T}}\in \textnormal{HC}^{i_U(H)}$ and $i_U(\tree{T})\in i_U(H)$. By the absoluteness of well-foundedness, $i_U(H)$ satisfies that $\bar{\tree{T}}$ is a weak hull of $i_U(\tree{T})$. The elementarity of $i_U\restriction H$ implies $H$ satisfies that $\bar{\tree{T}}$ is a weak hull of $\tree{T}$. This is $\Sigma_1$ so $\bar{\tree{T}}$ really is a weak hull of $\tree{T}$.
\end{proof}

The following is our main theorem.

\begin{thm}\label{thm - main v1}
Assume $\AD_\mathbb{R}+V=L(P(\mathbb{R}))+\textnormal{HPC}$. Let $U$ be an ultrafilter on an ordinal. Then there is a normal non-dropping ordinal definable iteration tree $\tree{T}$ of length $\Theta$ on $\HOD$ by $\Sigma_\HOD$ with a unique cofinal branch $b$ such that $M_b^\tree{T}=i_U(\HOD)$ and $i_b^\tree{T}=i_U\restriction \HOD$.
\end{thm}

The notation $\Sigma_\HOD$ is defined in the remarks following \cref{thm - pc generators}.

\begin{proof}
 Fix $\langle (H_\alpha, \Sigma_\alpha)\mid \alpha<\eta\rangle$ as in \cref{thm - pc generators}. Fix $\alpha<\eta$. Also fix a countable least branch hod pair $(P,\Sigma)$ with scope $\textnormal{HC}$ and a normal non-dropping iteration tree $\tree{S}$ on $P$ such that $(H_\alpha, \Sigma_\alpha)$ is an iterate of $(P, \Sigma^+\restriction V_\Theta)$ via $\tree{S}$. (See \cref{defn - sigma plus}.) By \cref{lemma - strategy preservation}, $i_U(\tree{S})$ is by $\Sigma^+$ so $i_U(H_\alpha)$ is a non-dropping iterate of $(P,\Sigma^+\restriction V_\Theta)$.
 
 Since $o(H_\alpha)$ is closed under ultrapowers in $V$ and the universe of $H_\alpha$ is a rank initial segment of $\HOD$, $i_U(H_\alpha)\subseteq H_\alpha$. By \cref{lemma - iterate criterion}, $(i_U(H_\alpha), (\Sigma^+\restriction V_\Theta)_{i_U(H_\alpha)})$ is a normal non-dropping iterate of $(H_\alpha, \Sigma_\alpha)$. 
 Let $\tree{T}_\alpha$ be the unique normal tree witnessing this. We claim that the main branch embedding $j$ of $\tree{T}_\alpha$ is equal to $i_U\restriction H_\alpha$. By \cref{lemma - commutativity}, $j\circ i_{0,\infty}^\tree{S}= i_{0,\infty}^{i_U(\tree{S})}$. Note that $i_U\circ i_{0,\infty}^\tree{S}= i_{0,\infty}^{i_U(\tree{S})}$ and therefore $j\restriction i_{0,\infty}^\tree{S}[P]= i_U\restriction i_{0,\infty}^\tree{S}[P]$. Since $o(H_\alpha)$ is an inaccessible cardinal in $\HOD$ and $j$ and $i_U\restriction H_\alpha$ are both in $\HOD$, $(H_\alpha, j, i_U\restriction H_\alpha)$ satisfies $\ZFC$. Applying \cref{thm - uniqueness} in this model, $j\restriction o(H_\alpha) = i_U\restriction o(H_\alpha)$. Since $H_\alpha=\Hull^{H_\alpha}(i_{0,\infty}^\tree{S}[P] \cup o(H_\alpha))$, it follows that $j=i_U\restriction H_\alpha$, as claimed. 

 For $\alpha<\beta<\eta$, since $(H_\alpha, \Sigma_\alpha)\mo^* (H_\beta, \Sigma_\beta)$ we can view $\tree{T}_\alpha$ as a non-dropping normal tree on $(H_\beta, \Sigma_\beta)$ with the same exit extenders and tree order.
By the uniqueness of normal trees (\cref{rmk - uniqueness of normal trees}), since $(H_\alpha, \Sigma_\alpha)\mo (H_\beta, \Sigma_\beta)$ and $i_U(H_\alpha)\mo i_U(H_\beta)$, $\tree{T}_\beta$ is an extension of $\tree{T}_\alpha$, viewed in this way. Let $(H,\Lambda)=\bigcup_{\alpha<\eta} (H_\alpha, \Sigma_\alpha)$. 
 Let $\tree{T}=\bigcup \tree{T}_\alpha$, viewed as a tree on $H$ by $\Lambda$. Note that $\tree{T}$ has length $\Theta$ and so does not have a last model. However, since $H_\alpha\mo^*H_\beta$ for $\alpha<\beta<\eta$, $\tree{T}$ is essentially a stack of normal trees $\langle \tree{U}_\alpha\mid \alpha<\eta\rangle$ on $H$: $\tree{U}_\alpha$ consists of the exit extenders of $\tree{T}$ with length between $\sup _{\beta<\alpha}o(H_\beta)$ and $o(H_\alpha)$. It follows that $\tree{T}$ has a unique cofinal branch $b$, obtained by concatenating the main branches in the stack $\langle \tree{U}_\alpha\mid \alpha<\eta\rangle$. 
 Moreover, $M_b^{\tree{T}}=\bigcup_{\alpha<\eta} M_\infty^{\tree{T}_\alpha}=\bigcup_{\alpha<\eta}i_U(H_\alpha)$. Also, $i_b^{\tree{T}}\restriction H_\alpha = i_{0,\infty}^{\tree{T}_\alpha} =i_U\restriction H_\alpha$. Therefore $i_b^{\tree{T}}=i_U\restriction H$. Here we just mean that $i_b^\tree{T}(x)= i_U(x)$ for all $x\in H$; $i_U(H)$ may be different from $M_b^\tree{T}$, in general, since it is possible that $i_U(\Theta)>\Theta$ when $\Theta$ is singular. (In any case, $M_b^\tree{T}=i_U(H)|\Theta$.)

Let $E$ be the extender of $i_b^{\tree{T}}$. Since $i_b^{\tree{T}}=i_U\restriction H$, $E$ is equal to the extender of length $\Theta$ derived from $i_U\restriction \HOD$. From now on let us consider $\tree{T}$ as a tree on $\HOD=L[H]$. Note that $M_b^\tree{T}=\Ult(\HOD, E)$ which is equal to $i_U(\HOD)$ by \cref{lemma - generators}. Finally, $i_b^\tree{T}=i_E^\HOD=i_U\restriction \HOD$, again by \cref{lemma - generators}.
\end{proof}

We can use a variation of this argument to prove a stronger result, \cref{thm - main v2}, which generalizes \cref{thm - main v1} to ultrafilters that are not on ordinals, important objects of study in determinacy (the Martin measure and strong partition measures are such ultrafilters, for example). This will involve replacing the appeal to \cref{thm - uniqueness} with a more detailed analysis of how the models $H_\alpha$ are obtained as direct limits. 

We need the following result due to Schlutzenberg (essentially \cite[Lemma 5.2]{Farmer}) and Siskind, independently.

\begin{thm}[{Schlutzenberg \cite{Farmer}, Siskind}]\label{thm - least common iterate}
   Let $D$ be a countable directed set of non-dropping iterates of a least branch hod pair $(P,\Sigma)$ with scope $\textnormal{HC}$. Let $M_\infty$ be the direct limit of $D$. Then $M_\infty$ is the least common iterate of all $Q\in D$. More precisely, $M_\infty$ is a non-dropping $\Sigma_Q$-iterate of every $Q\in D$, and for any $N$ which is a non-dropping $\Sigma_Q$-iterate of every $Q\in D$, $N$ is a non-dropping $\Sigma_{M_\infty}$-iterate of $M_\infty$.
\end{thm}

\begin{proof}
    Let $\vec{\tree{S}}$ be a length $\omega$ stack of countable normal trees on $(P,\Sigma)$ such that $\{ P_n\mid n<\omega\}$ is cofinal in $D$, where $P_n=M_0^{\tree{S}_n}$. In particular, the direct limit $M_\infty$ is equal to $M_b^{\vec{\tree{S}}}$ for $b$ the unique cofinal branch of $\vec{\tree{S}}$.  For $m\leq n\leq \omega$, let $\tree{X}_{m,n}=X(\vec{\tree{S}}\restriction[m,n))$.

    Fix $N$ a common iterate of all $Q\in D$. Let $\tree{U}$ and $\tree{V}$ be the padded normal trees of the comparison of $N$ and $M_\infty$. By \cref{lemma - directedness}, the trees $\tree{U}$ and $\tree{V}$ are non-dropping and have a common final model. So it suffices to show that $\tree{U}$ is trivial. 
    
    Towards a contradiction, suppose $E=E_\xi^\tree{U}$ is the least non-trivial extender used in $\tree{U}$. Let $\tree{X}=X(\tree{X}_{0,\omega}, \tree{V}\restriction\xi+1)$ and $\Psi:\tree{X}_{0,\omega}\to\tree{X}$ the associated weak tree embedding. As in the proof of \cref{lemma - iterate criterion}, using that $N$ is an iterate of $(P,\Sigma)$, we get that $E$ must have been used in $\tree{X}$ and that $\tree{V}\restriction\xi+1$ doesn't drop along its main branch. Let $\alpha+1<\lh(\tree{X})$ be such that $E_\alpha^\tree{X}= E$. Since $\lh(E)>\lh(E_\eta^{\tree{V}})$ for all $\eta+1<\xi$, \cref{lemma - tree emb} implies $\alpha\in \ran(u^\Psi)$. Let $\bar{\alpha}=(u^{\Psi})^{-1}(\alpha)$. \cref{lemma - tree emb} also implies that $i_{0,\xi}^{\tree{V}}$ restricts to a cofinal elementary map from $M_{\bar{\alpha}}^{\tree{X}_{0,\omega}}|\lh(E_{\bar{\alpha}}^{\tree{X}_{0,\omega}})$ into $M_\alpha^{\tree{X}}|\lh(E)$. In particular, $E_{\bar{\alpha}}^{\tree{X}_{0,\omega}}=(i_{0,\xi}^{\tree{V}})^{-1}(E)$.
    
    Now, by \cref{prop - infinite stack}, we may let $n<\omega$ be such that $\bar{\alpha}\in \ran(u^{\Phi^{\vec{\tree{S}}}_{n,\omega}})$. Let $\tree{Y}= X(\tree{X}_{n,\omega}, \tree{V}\restriction\xi+1)$ and let $\Phi:\tree{X}_{n,\omega}\to \tree{Y}$ be the associated weak tree embedding. 
    Again, as in the proof of \cref{lemma - iterate criterion}, but now using that $N$ is an iterate of $P_n$, we get that $E=E_{\beta}^\tree{Y}$ for some $\beta\in \ran(u^\Phi)$, and letting $\bar{\beta}=(u^\Phi)^{-1}(\beta)$, $E_{\bar{\beta}}^{\tree{X}_{n,\omega}}=(i_{0,\xi}^{\tree{V}})^{-1}(E)$. 
    Therefore, $E_{\bar{\beta}}^{\tree{X}_{n,\omega}}=E_{\bar{\alpha}}^{\tree{X}_{0,\omega}}$. 
    We also have that $\tree{X}_{0,\omega}= X(\tree{X}_{0,n}, \tree{X}_{n,\omega})$ and that $\Phi^{\vec{\tree{S}}}_{n,\omega}$ is actually the weak tree embedding arising from this normalization, by \cref{prop - infinite stack commutativity}. 
    So the ``moreover" clause of \cref{lemma - drops} gives that $\bar{\alpha}\not\in \ran (u^{\Phi^{\vec{\tree{S}}}_{n,\omega}})$, contradicting our choice of $n$.
\end{proof}

A reflection argument lets us extend this theorem to arbitrary directed sets $D$ of non-dropping iterates of a least branch hod pair $(P,\Sigma)$ with scope $\textnormal{HC}$. The direct limit of such a set $D$ may be uncountable, and so cannot be a $\Sigma$-iterate of $P$, but it will be a $\Sigma^+$-iterate of $P$ and will still have an analogous minimality property.

\begin{cor}[{$\textnormal{DC}_\mathbb{R}$}]\label{cor - iteration criterion 2}
   Let $D$ be a directed set of non-dropping iterates of a least branch hod pair $(P,\Sigma)$ with scope $\textnormal{HC}$. Let $M_\infty$ be the direct limit of $D$. Then $M_\infty$ is a $\Sigma_Q^+$-iterate of all $Q\in D$ and for any $N$ which is a $\Sigma^+_Q$-iterate of all $Q\in D$, $N$ is a $\Sigma^+_{M_\infty}$-iterate of $M_\infty$.\footnote{$\Sigma^+_{M_\infty}$ is defined in the discussion following \cref{defn - sigma plus}.}
\end{cor}
\begin{proof}
   Fix $N$ a $\Sigma^+_Q$-iterate of all $Q\in D$. Let $\tree{T}_Q$ be the unique non-dropping normal tree on $Q$ by $\Sigma^+_Q$ with last model $N$. By passing to $L(\mathbb{R}, \Sigma, Q\mapsto \tree{T}_Q)$, we may assume $\textnormal{DC}$. 
   
   Let $Y$ be a countable elementary substructure of a sufficiently large $V_\gamma$ such that $P$, $\Sigma$, $D$, and $N$ are all in $Y$. Let $\pi: H\to Y$ be the inverse of the transitive collapse. Let $\bar{\Sigma}=\pi^{-1}(\Sigma)$, $\bar{D}=\pi^{-1}(D)$, $\bar{N}=\pi^{-1}(N)$. 
    Then $\bar{\Sigma}=\Sigma\cap \textnormal{HC}^H$ by strong hull condensation and $\bar{D}=D\cap \textnormal{HC}^H$, so $\bar{D}$ is a countable directed set of non-dropping iterates of $(P,\Sigma)$. Also, $\bar{N}$ is a common non-dropping $(\Sigma^+_Q)^H$-iterate of all $Q\in \bar{D}$. For $Q\in \bar{D}$, let $\bar{\tree{T}}_Q\in H$ be the unique normal tree on $Q$ by $(\Sigma^+_Q)^H$ witnessing that $\bar{N}$ is a non-dropping iterate of $Q$. Then $\pi(\bar{\tree{T}}_Q)=\tree{T}_Q$, by the elementarity of $\pi$. 
So $\bar{\tree{T}}_Q$ is a countable weak hull of $\tree{T}_Q$, which implies that $\bar{\tree{T}}_Q$ is by $\Sigma_Q$, since $\tree{T}_Q$ is by $\Sigma^+_Q$. Therefore  $\bar{N}$ is a non-dropping $\Sigma_Q$ iterate of every $Q\in \bar{D}$. 

Let $\bar{M}_\infty$ denote the direct limit of $\bar{D}$. Let $\tree{S}$ be the countable non-dropping tree on $(P,\Sigma)$ with last model $\bar{M}_\infty$. Note that $\tree{S}\in H$ since by elementarity $H$ satisfies that $\bar{M}_\infty$ is a $\Sigma^+$-iterate of $P$. By \cref{thm - least common iterate}, there is a countable non-dropping tree $\tree{V}$ on $(\bar{M}_\infty, \Sigma_{\bar{M}_\infty})$ with last model $\bar{N}$. We claim that $\tree{V}\in H$. We'll show by induction that every initial segment of $\tree{V}$ is in $H$. Since $\tree{V}$ arises from the comparison of $\bar{M}_\infty$ and $\bar{N}$ by least extender disagreement and these models are in $H$, we just need to show that if $\lambda$ is a limit ordinal and $\tree{V}\restriction \lambda\in H$, then the branch $\Sigma_{\bar{M}_\infty}(\tree{V}\restriction\lambda)$ is in $H$. (The issue here is that $\lambda$ may not be countable in $H$ and we cannot assume that $H$ is closed under $\Sigma$. If $H$ were closed under $\Sigma$ we could conclude that $\Sigma^+$ is defined on all trees in $V_\gamma$, which will not be true in the case of interest.) 
 
 Let $\tree{X}=X(\tree{S},\tree{V}\restriction\lambda+1)$  and let $\nu$ be the supremum of the stages where exit extenders of $\tree{V}\restriction\lambda$ are used in $\tree{X}$. That is, for $\xi<\lambda$, let $\alpha_\xi+1<\lh(\tree{X})$ be such that $E_{\alpha_\xi}^{\tree{X}}=E_\xi^\tree{V}$ and $\nu=\sup_{\xi<\lambda} \alpha_\xi$. Since $X(\tree{S},\tree{V})=\bar{\tree{T}}_P$ and $\tree{X}\restriction\nu=X(\tree{S},\tree{V})\restriction\nu$,  $\tree{X}\restriction\nu$ is an initial segment of $\bar{\tree{T}}_P$. Let $b=\Sigma(\tree{X}\restriction\nu)$. By \cref{prop - branch uniqueness}, $\Sigma(\tree{V}\restriction\lambda)$ is the unique cofinal branch $c$ of $\tree{V}\restriction \lambda$ such that $(\tree{X}\restriction\nu)^\frown b\moq X(\tree{S},(\tree{V}\restriction\lambda)^\frown c)$. Now by an absoluteness argument, it follows from this characterization that $\Sigma(\tree{V}\restriction\lambda)$ is in $H$. Namely, let $G\subseteq \textnormal{Col}(\omega, \lambda)$ be an $H$-generic filter. In $H[G]$ by $\Sigma^1_1$-absoluteness there is a cofinal branch $c$ of $\tree{V}\restriction \lambda$ such that $(\tree{X}\restriction\nu)^\frown b\moq X(\tree{S},(\tree{V}\restriction\lambda)^\frown c)$. By uniqueness $c=\Sigma(\tree{V}\restriction \lambda)$. Since $\Sigma(\tree{V}\restriction \lambda)$ belongs to every $\textnormal{Col}(\omega, \lambda)$-generic extension of $H$, $\Sigma(\tree{V}\restriction \lambda)$ belongs to $H$, as desired. 
 
 This proves that $\tree{V}$ is in $H$. Therefore, by very strong hull condensation for $\Sigma_{\bar{M}_\infty}$ and full normalization for $\Sigma$, $H$ satisfies that $\tree{V}$ is by $\Sigma^+_{\bar{M}_\infty}$. So in $H$, $\bar{N}$ is a non-dropping $\Sigma^+_{\bar{M}_\infty}$-iterate of $\bar{M}_\infty$, and so by elementarity, $N$ is a $\Sigma^+_{M_\infty}$-iterate of $M_\infty$.
\end{proof}

Next we state a refinement of \cref{thm - pc generators}.

\begin{thm}[{Steel, \cite{mousepairs}}]\label{thm - pc generators 2}
    Assume $\AD_\mathbb{R}$ and  \textnormal{HPC}. Then there is a sequence $\langle (H_\alpha,\Sigma_\alpha)\mid\alpha<\eta\rangle$ of least branch hod pairs such that for every $\alpha<\eta$,
    \begin{enumerate}
    \item  the universe of $H_\alpha$ is $\HOD\cap V_{o(H_\alpha)}$ and $o(H_\alpha)$ is strongly inaccessible in $\HOD$ and closed under ultrapowers in $V$,
        \item\label{item - pc generators}  there is a countable least branch hod pair $(P, \Sigma)$ with scope $\textnormal{HC}$ such $(H_\alpha, \Sigma_\alpha)$ is an iterate of $(P, \Sigma^+\restriction V_\Theta)$
        \item for every $\alpha<\beta<\eta$, $(H_\alpha,\Sigma_\alpha)\mo^* (H_\beta, \Sigma_\beta)$,
        %define cutpoint initial segment
    \end{enumerate}
Finally, letting $H$ be the least branch premouse $\bigcup_{\alpha<\eta} H_\alpha$, we have $o(H)=\Theta$ and $L[H]=\HOD$.
\end{thm}

As in the remarks following \cref{thm - pc generators}, we let $\Sigma_\HOD$ denote the partial iteration strategy for $\HOD$ coming from \cref{thm - pc generators 2}.

Our generalization of \cref{thm - main v1}  is the following.

\begin{thm}\label{thm - main v2}
Assume $\AD_\mathbb{R}+V=L(P(\mathbb{R}))+\textnormal{HPC}$. Let $U$ be an ultrafilter on a set $X\leq^* \mathbb{R}$ such that $\Ult(\Ord, U)$ is well-founded. 
Then there is a normal non-dropping iteration tree $\tree{T}$ of limit length on $\HOD$ by $\Sigma_\HOD$ with a unique cofinal branch $b$ such that $M_b^\tree{T}=i_U(\HOD)$ and $i_b^\tree{T}=i_U\restriction \HOD$.
\end{thm}

\begin{proof}
     Fix $\langle (H_\alpha, \Sigma_\alpha)\mid \alpha<\eta\rangle$ as in \cref{thm - pc generators 2}. Fix $\alpha<\eta$. Also fix a countable least branch hod pair $(P,\Sigma)$ with scope $\textnormal{HC}$ such that $(H_\alpha, \Sigma_\alpha)$ is an iterate of $(P, \Sigma^+\restriction V_\Theta)$ and $H_\alpha=M_\infty(P,\Sigma)$.
     For each countable non-dropping iterate $Q$ of $(P,\Sigma)$, let $\tree{S}_Q$ the unique non-dropping tree on $(Q, \Sigma_Q^+\restriction V_\Theta)$ with last model $H_\alpha$. By \cref{lemma - strategy preservation}, $i_U(\tree{S}_Q)$ is by $\Sigma^+_Q$ so $i_U(H_\alpha)$ is a non-dropping $\Sigma^+_Q$-iterate of $Q$.  By \cref{cor - iteration criterion 2}, $i_U(H_\alpha)$ is a non-dropping $\Sigma_\alpha^+$-iterate of $H_\alpha$. Let $\tree{T}_\alpha$ be the unique normal tree witnessing this. 
     
     We claim that the main branch embedding $j$ of $\tree{T}_\alpha$ is equal to $i_U\restriction H_\alpha$. By \cref{lemma - commutativity} and a Skolem hull argument, for any countable non-dropping iterate $Q$ of $(P,\Sigma)$, $j\circ i_{0,\infty}^{\tree{S}_Q}= i_{0,\infty}^{i_U(\tree{S}_Q)}$. Note that $i_U\circ i_{0,\infty}^{\tree{S}_Q}= i_{0,\infty}^{i_U(\tree{S}_Q)}$ and therefore $j\restriction i_{0,\infty}^{\tree{S}_Q}[Q]= i_U\restriction i_{0,\infty}^{\tree{S}_Q}[Q]$. By the definition of the direct limit, $H_\alpha=\bigcup_Q i_{0,\infty}^{\tree{S}_Q}[Q]$, and it follows that $j=i_U\restriction H_\alpha$, as claimed. 

 For $\alpha<\beta<\eta$, since $(H_\alpha, \Sigma_\alpha)\mo^* (H_\beta, \Sigma_\beta)$ we can view $\tree{T}_\alpha$ as a non-dropping normal tree on $(H_\beta, \Sigma_\beta)$ with the same exit extenders and tree order.
 By the uniqueness of normal trees (\cref{rmk - uniqueness of normal trees}), since $(H_\alpha, \Sigma_\alpha)\mo (H_\beta, \Sigma_\beta)$ and $i_U(H_\alpha)\mo i_U(H_\beta)$, $\tree{T}_\beta$ is an extension of $\tree{T}_\alpha$, viewed in this way. Let $(H,\Lambda)=\bigcup_{\alpha<\eta} (H_\alpha, \Sigma_\alpha)$. 
 Let $\tree{T}=\bigcup \tree{T}_\alpha$, viewed as a tree on $H$ by $\Lambda$. Note that $\tree{T}$ has limit length and so does not have a last model. However, since $H_\alpha\mo^*H_\beta$ for $\alpha<\beta<\eta$ (\cref{defn - cutpoint}), $\tree{T}$ is essentially a stack of normal trees $\langle \tree{U}_\alpha\mid \alpha<\eta\rangle$ on $H$: $\tree{U}_\alpha$ consists of the exit extenders of $\tree{T}$ with length between $\sup _{\beta<\alpha}o(H_\beta)$ and $o(H_\alpha)$. It follows that $\tree{T}$ has a unique cofinal branch $b$, obtained by concatenating the main branches in the stack $\langle \tree{U}_\alpha\mid \alpha<\eta\rangle$. 
 Moreover, $M_b^{\tree{T}}=\bigcup_{\alpha<\eta} M_\infty^{\tree{T}_\alpha}=\bigcup_{\alpha<\eta}i_U(H_\alpha)$. Also, $i_b^{\tree{T}}\restriction H_\alpha = i_{0,\infty}^{\tree{T}_\alpha} =i_U\restriction H_\alpha$. Therefore $i_b^{\tree{T}}=i_U\restriction H$. Here we just mean that $i_b^\tree{T}(x)= i_U(x)$ for all $x\in H$; $i_U(H)$ may be different from $M_b^\tree{T}$, in general, since it is possible that $i_U(\Theta)>\sup i_U[\Theta]$ when $\Theta$ is singular. (In any case, $M_b^\tree{T}=i_U(H)|\sup i_U[\Theta]$.)

Let $E$ be the extender of $i_b^{\tree{T}}$. Since $i_b^{\tree{T}}=i_U\restriction H$, $E$ is equal to the extender of length $\sup i_U[\Theta]$ derived from $i_U\restriction \HOD$. From now on let us consider $\tree{T}$ as a tree on $\HOD=L[H]$. Note that $M_b^\tree{T}=\Ult(\HOD, E)$ which is equal to $i_U(\HOD)$ by \cref{lemma - generators}. Finally, $i_b^\tree{T}=i_E^\HOD=i_U\restriction \HOD$, again by \cref{lemma - generators}.
\end{proof}

\subsection{$V = L(\mathbb{R})$}\label{section - L(R)}

In this section we'll prove our main theorem in $L(\mathbb{R})$ under determinacy using arguments similar to those of the previous section. While the full $\HOD$ of $L(\mathbb{R})$ can be seen to be of the form $L[M_\infty(P,\Sigma)]$ for some mouse pair $(P,\Sigma)$ (the rigidly layered version of $M_\omega$ with its strategy, for example), there is no such mouse pair which is actually a member of $L(\mathbb{R})$. To get around this, we will use a reflection argument. 

We will need the following result which follows from the proof of \cref{thm - main v2}.
\begin{lma}\label{lemma - ultrapowers of M-infinity}
    Assume $\AD^++V=L(P(\mathbb{R}))$. Let $U$ be an ultrafilter on an ordinal. Let $(P,\Sigma)$ be a mouse pair with scope $\textnormal{HC}$ and let $M_\infty= M_\infty(P,\Sigma)$. Then there is a unique normal tree $\tree{V}$ on $M_\infty$ by $\Sigma^+_{M_\infty}$ with last model $i_U(M_\infty)$ and $i^\tree{V}_{0,\infty}= i_U\restriction M_\infty$.
\end{lma}

(We will not use the full generality of \cref{lemma - ultrapowers of M-infinity} and for the $(P,\Sigma)$ we need to consider, we could instead use the argument from \cref{thm - main v1}.)

Here is our main theorem for $L(\mathbb{R})$.

\begin{thm}\label{thm - main L(R)}
    Assume $\AD+V=L(\mathbb{R})$. Let $U$ be an ultrafilter on an ordinal. Then there is a normal non-dropping ordinal definable iteration tree $\tree{T}$ on $\HOD$ of length $\Theta$ based on $\HOD|\Theta$ by the short tree strategy for $\HOD$ with a unique cofinal branch $b$ such that $M_b^\tree{T}=i_U(\HOD)$ and $i_b^\tree{T}=i_U\restriction \HOD$.
\end{thm}

\begin{proof}
    Because the short tree strategy of $\HOD$ is definable (in $L(\mathbb{R})$), the theorem statement can be expressed by a first-order sentence $\psi_0$ in the language of set theory. We would like to show $L(\mathbb{R})\vDash \psi_0$. We will do this via a reflection argument following \cite[\S 7]{hodas}.

    Let $\gamma$ be least such that $L_\gamma(\mathbb{R})$ satisfies \[\textnormal{ZF}^- +``P(P(\mathbb{R})) \textnormal{ exists}" + \neg \psi_0.\] Let $\theta=\Theta^{L_\gamma(\mathbb{R})}$. The argument of \cite[\S 7]{hodas} produces a pure extender mouse pair $(P,\Sigma)$ with the following properties.\footnote{In \cite[\S 7]{hodas}, the analogous mouse pair is called $(M_0, \Sigma_0)$.} First, $P$ has $\omega$ Woodin cardinals. Second, the following hold, where $\delta_0$ is the least Woodin cardinal of $P$, $P_0=P|\delta_0$, and $M_\infty=M_\infty(P^-,\Sigma_{P_0})$:
    %let $\lambda_\infty$ be the supremum of the Woodin cardinals of $M_\infty$, and let $F$ be the set of normal trees of limit length  in $V_{\lambda_\infty}\cap M_\infty$ based on $M_\infty|\theta$.\footnote{The ordinal $\theta$ is the least Woodin cardinal of $M_\infty$.} Let $\Sigma_\infty = \Sigma^+_{M_\infty}\restriction F$.
    \begin{enumerate}
        \item $V_\theta\cap\HOD^{L_\gamma(\mathbb{R})}=M_\infty$, and
       % \item $\HOD^{L_\gamma(\mathbb{R})} = L_\gamma(M_\infty, \Sigma_\infty)$,
        \item there is a unique normal tree $\tree{S}$ on $P_0$ of length $\theta+1$ by $\Sigma^+_{P_0}$ with $M_\infty= M_\theta^{\tree{S}}$, and $\tree{S}\restriction \theta\in L_\gamma(\mathbb{R})$. 
    \end{enumerate} 

    Let $U\in L_\gamma(\mathbb{R})$ be an $L_\gamma(\mathbb{R})$-ultrafilter on $\kappa<\theta$ witnessing the failure of $\psi_0$. The Coding Lemma implies that $P(\kappa)\subseteq L_\gamma(\mathbb{R})$, so $U$ is an ultrafilter in $L(\mathbb{R})$. Using the minimality of $\gamma$, we will show that $\Ult_0(L_\gamma(\mathbb{R}), U)=i_U(L_\gamma(\mathbb{R}))$ and $$i_U^{L_\gamma(\mathbb{R})}=i_U\restriction L_\gamma(\mathbb{R}).$$
    
    For every $n$, let $H_n= \textnormal{Hull}_{\Sigma_n}^{L_\gamma(\mathbb{R})}(\mathbb{R})$. By Replacement, every $H_n$ is a member of $L_\gamma(\mathbb{R})$ and is the surjective image of $\mathbb{R}$ in $L_\gamma(\mathbb{R})$. The Coding Lemma implies that every $\kappa$-sequence of sets of reals with Wadge rank bounded below $\theta$ belongs to $L_\gamma(\mathbb{R})$. It follows that every partial function from $\kappa$ into $H_n$ belongs to $L_\gamma(\mathbb{R})$. Fix a function $f:\kappa\to L_\gamma(\mathbb{R})$. We'll show that $f\restriction A\in L_\gamma(\mathbb{R})$ for some $A\in U$. Let $A_n=\{\alpha<\kappa\mid f(\alpha)\in H_n\}$. By the countable completeness of $U$, for some $n$, $A_n\in U$. Since $f\restriction A_n$ is a partial function from $\kappa$ into $H_n$, $f\restriction A_n\in L_\gamma(\mathbb{R})$, as desired. This proves our claim that $\Ult_0(L_\gamma(\mathbb{R}), U)=i_U(L_\gamma(\mathbb{R}))$ and $i_U^{L_\gamma(\mathbb{R})}=i_U\restriction L_\gamma(\mathbb{R})$.

    By \cref{lemma - ultrapowers of M-infinity}, there is a unique normal tree $\tree{V}$ on $M_\infty$ by $\Sigma^+_{M_\infty}$ with last model $i_U(M_\infty)$ and $$i^\tree{V}_{0,\infty}=i_U\restriction M_\infty.$$
Since $i_U(\tree{S}\restriction\theta)\in L_\gamma(\mathbb{R})$ and $i_U(\tree{S})=X(\tree{S}, \tree{V})$, we can use \cref{prop - branch uniqueness} to show that every proper initial segment of $\tree{V}$ is in $L_\gamma(\mathbb{R})$. More precisely, since $\tree{V}$ is the tree of the comparison of $M_\infty$ and $i_U(M_\infty)$, it suffices to show that for all limit ordinals $\lambda<\lh(\tree{V})$, $[0,\lambda)_\tree{V}$ is in $L_\gamma(\mathbb{R})$ and this follows from \cref{prop - branch uniqueness} by the $\Sigma^1_1$-absoluteness argument used in the proof of \cref{cor - iteration criterion 2}.

    Finally, since $i_{0,\infty}^\tree{V}=i_U\restriction M_\infty= i_U^{L_\gamma(\mathbb{R})}\restriction M_\infty$ is definable over $L_\gamma(\mathbb{R})$, the main branch of $\tree{V}$ is a member of $L_\gamma(\mathbb{R})$. Let $\tree{T}$ be the tree on $\HOD^{L_\gamma(\mathbb{R})}$ with the same extenders and tree-order as $\tree{V}$. By \cref{lemma - generators} applied in $L_\gamma(\mathbb{R})$, $i_{0,\infty}^\tree{T}=i^{L_\gamma(\mathbb{R})}_U\restriction \HOD^{L_\gamma(\mathbb{R})}$. This contradicts that $U$ witnessed the failure of $\psi_0$ in $L_\gamma(\mathbb{R})$.
\end{proof}

If there is a fully iterable $M_\omega^\#$, then we also have access to the external characterization of $\HOD^{L(\mathbb{R})}$. This requires a bit of notation. Let $\Sigma$ be the iteration strategy for $M_\omega$ coming from the unique iteration strategy for $M_\omega^\#$. Let $\delta_0$ be the least Woodin cardinal of $M_\omega$. Let $M_\infty$ be the direct limit of all non-dropping iterates of $M_\omega$ by $\Sigma$ via countable non-dropping stacks of normal trees based on $M_\omega|\delta_0$. Also let $\delta_\infty$ be the least Woodin cardinal of $M_\infty$ and $\lambda_\infty$ be the supremum of the Woodin cardinals of $M_\infty$. Finally, let $\Lambda$ be the restriction of $\Sigma_{M_\infty}$ to stacks of normal trees based on $M_\infty|\delta_\infty$ which are members of $M_\infty|\lambda_\infty$.

\begin{thm}[Steel-Woodin \cite{hodas}]\label{thm - hod L(R) v2}
Assume $M_\omega^\#$ exists and is fully iterable. Then $(V_\Theta\cap \HOD)^{L(\mathbb{R})}$ is the universe of $M_\infty|\delta_\infty$ and $\HOD^{L(\mathbb{R})}=L[M_\infty, \Lambda]$.    
\end{thm}

In particular, if $M_\omega^\#$ exists and is fully iterable, the $H$ from the statement of \cref{thm - hod L(R)} is just $M_\omega$. 

That $M_\omega^\#$ exists and is fully iterable implies that $\mathbb{R}^\#$. If we assume additionally that $L(\mathbb{R}, \mathbb{R}^\#)\vDash \AD$, we can strengthen \cref{thm - main L(R)} a bit to say that the trees we produce are actually according to tails of $\Sigma$, the iteration strategy for $M_\omega$.

\begin{thm}\label{thm - main L(R) v2}
    Assume $M_\omega^\#$ exists and is fully iterable and $L(\mathbb{R}, \mathbb{R}^\#)\vDash \AD$. Let $U$ be an $L(\mathbb{R})$-ultrafilter on an ordinal and let $\tree{T}^\frown b$ be the tree as in \cref{thm - main L(R)}. Then, $\tree{T}^\frown b$ is by $\Sigma_{M_\infty}$ (viewing $\tree{T}^\frown b$ as an iteration tree on $M_\infty$).
\end{thm}

\begin{proof}[Proof sketch.]
Let $U$ be an $L(\mathbb{R})$-ultrafilter on an ordinal and let $\tree{T}^\frown b$ be the tree on $\HOD^{L(\mathbb{R})}$ as in \cref{thm - main L(R)}. By a Coding Lemma argument, $U$ is still an ultrafilter in $L(\mathbb{R}, \mathbb{R}^\#)$, $i^{L(\mathbb{R}, \mathbb{R}^\#)}_U((\HOD|\Theta)^{L(\mathbb{R})}) = i_U^{L(\mathbb{R})}((\HOD|\Theta)^{L(\mathbb{R})})$, and 
$i^{L(\mathbb{R}, \mathbb{R}^\#)}_U\restriction(\HOD|\Theta)^{L(\mathbb{R})} = i_U^{L(\mathbb{R})}\restriction(\HOD|\Theta)^{L(\mathbb{R})}$. (This uses that $\Theta^{L(\mathbb{R})}$ has countable cofinality in $L(\mathbb{R}, \mathbb{R}^\#)$.) 
We also have that the restriction of $\Sigma$ to countable stacks of normal trees on $M_\omega|\delta_0$, which we denote $\Lambda$, is a member of $L(\mathbb{R}, \mathbb{R}^\#)$ because the full strategy for countable stacks on $M_\omega^\#$ is in $L(\mathbb{R}, \mathbb{R}^\#)$. We can apply \cref{lemma - ultrapowers of M-infinity} to $(M_\omega|\delta_0, \Lambda)$ in $L(\mathbb{R}, \mathbb{R}^\#)$ to get a non-dropping normal tree $\tree{V}$ by $\Lambda^+_{M_\infty|\delta_\infty}$ on $M_\infty|\delta_\infty$ with last model $i^{L(\mathbb{R})}_U(M_\infty|\delta_\infty)$ such that $i^\tree{V}_{0,\infty}=i^U_{L(\mathbb{R}, \mathbb{R}^\#)}\restriction M_\infty|\delta_\infty = i^U_{L(\mathbb{R})}\restriction M_\infty|\delta_\infty$, 
using that $M_\infty|\delta_\infty =(\HOD|\Theta)^{L(\mathbb{R})}$, by \cref{thm - hod L(R) v2}, and our observation that $i^U_{L(\mathbb{R}, \mathbb{R}^\#)}$ and $i^U_{L(\mathbb{R})}$ agree on $(\HOD|\Theta)^{L(\mathbb{R})}$. These observations also imply that $\tree{V}=\tree{T}^\frown b$, viewed as a tree on $M_\infty|\delta_\infty$. Since the full strategy $\Sigma$ for $M_\omega$ has very strong hull condensation and fully normalizes well, it follows that $\tree{T}^\frown b$ must actually be by $\Sigma_{M_\infty}$, as desired. 
\end{proof}

\bibliographystyle{plain}
\bibliography{Bibliography.bib}
\end{document}